\documentclass[12pt,a4paper,reqno]{amsart}

\usepackage[T1]{fontenc}
\usepackage[utf8]{inputenc}
\usepackage{lmodern}
\usepackage{microtype}
\usepackage{amsmath,amssymb,amsthm,mathtools}
\usepackage[pdftex,pdfpagelabels]{hyperref}
\hypersetup{
  hidelinks,
  pdftitle={Extremal densities for forbidden configurations in S-smooth numbers},
  pdfauthor={Nikola Veselinov}
}
\usepackage[nameinlink,noabbrev]{cleveref}
\usepackage{xcolor}
\usepackage{enumitem}
\usepackage{booktabs}

\usepackage{comment}

\numberwithin{equation}{section}

\usepackage{aliascnt}

\newtheorem{theorem}{Theorem}[section]

\newaliascnt{proposition}{theorem}
\newtheorem{proposition}[proposition]{Proposition}
\aliascntresetthe{proposition}

\newaliascnt{lemma}{theorem}
\newtheorem{lemma}[lemma]{Lemma}
\aliascntresetthe{lemma}

\newaliascnt{corollary}{theorem}
\newtheorem{corollary}[corollary]{Corollary}
\aliascntresetthe{corollary}

\newaliascnt{example}{theorem}

\aliascntresetthe{example}

\newaliascnt{conjecture}{theorem}

\aliascntresetthe{conjecture}

\theoremstyle{definition}
\newaliascnt{defn}{theorem}
\newtheorem{defn}[defn]{Definition}
\crefname{corollary}{Corollary}{Corollaries}
\aliascntresetthe{defn}

\theoremstyle{remark}
\newaliascnt{remark}{theorem}

\aliascntresetthe{remark}

\newcommand{\Zg}{\mathbb{Z}_{\geq0}}
\newcommand{\Z}{\mathbb{Z}}
\newcommand{\R}{\mathbb{R}}
\newcommand{\Q}{\mathbb{Q}}
\newcommand{\1}{\mathbf{1}}
\newcommand{\cS}{c_S}
\newcommand{\kS}{\kappa_S}
\newcommand{\FS}{F_S}
\newcommand{\GS}{G_S}
\newcommand{\PsiS}{\Psi_S}
\newcommand{\SigmaS}{\Sigma_S}
\newcommand{\fS}{f_S}
\newcommand{\alphaS}{\alpha_S}
\newcommand{\Deltaf}{\Delta f}
\newcommand{\ds}{d^{(S)}}

\newcommand{\Ss}{\texorpdfstring{$S$}{S}}

\title{Extremal densities for forbidden configurations in $S$-smooth numbers}
\author{Nikola Veselinov}
\address{Independent researcher}
\email{nikola.veselinov.veselinov@gmail.com}
\subjclass[2020]{11B75, 11N25, 05D05}
\keywords{smooth numbers, forbidden configurations, extremal density, combinatorial number theory, pattern avoidance}

\begin{document}

\begin{abstract}
  Let $S = \{p_1,\dots,p_r\}$ be a finite set of distinct primes, let $\PsiS(X)$ be the number of $S$-smooth integers not exceeding $X$, and let $\FS(X)$ be the maximum size of a subset of $M(S) \cap [1,X]$ containing no set $\{n,p_1 n,\dots,p_r n\}$. We prove that
  $
    \FS(X)=\frac{r}{r+1}\PsiS(X)+O_S\bigl((\log X)^{r-1}\bigr)
  $
  as $X \to \infty$, and equivalently that
  $
    \fS(k)=\frac{r}{r+1}k+O_S\bigl(k^{(r-1)/r}\bigr)
  $
  for the corresponding extremal function on the first $k$ $S$-smooth numbers. We also relate this problem to the analogous extremal problem on the full interval $[1,N]$. Using the classical theory of such forbidden configurations, we obtain a representation of the corresponding density constant $\alpha_S$ in terms of the increments of $f_S$, along with nested computable bounds and a recursive formula for the reciprocal tail over $S$-smooth numbers. We further show that rational reciprocal sums over $S$-smooth denominators need not arise from eventually periodic binary sequences. In the classical case $S=\{2,3\}$, we derive an explicit tail formula and prove two structural propositions for optimal sets.
\end{abstract}

\maketitle

\section{Introduction}

The extremal problem
\[
  F(N)\coloneq \max\bigl\{|A|:A\subseteq\{1,\dots,N\},\ A\text{ contains no set }\{x,2x,3x\}\bigr\}
\]
goes back to Graham, Witsenhausen, and Spencer \cite{GWS}, who proved the existence of the limit
\[
  \alpha\coloneq \lim_{N\to\infty}\frac{F(N)}{N}.
\]
More generally, density questions for sets avoiding or meeting prescribed multiplicative configurations were studied by Lucht \cite{Lucht} and by Chung, Erd\H{o}s, and Graham \cite{CEG}; see also Wang \cite{Wang}, Reznick and Holzsager \cite{RH}, Lai \cite{Lai}, and Leung and Wei \cite{LeungWei}. For background on smooth numbers we refer to Tenenbaum \cite{Tenenbaum}.

\smallskip

Fix a finite set of distinct primes $S=\{p_1,\dots,p_r\}$, $r\ge 2$, and write
\[
  M(S)\coloneq \{p_1^{a_1}\cdots p_r^{a_r}: a_1,\dots,a_r\in \Zg\}
\]
for the multiplicative semigroup generated by $S$. The elements of $M(S)$ are the $S$-smooth numbers, and
\[
  \PsiS(X)\coloneq |M(S)\cap [1,X]|
\]
denotes their counting function.

\smallskip

The main purpose of the present paper is to study the corresponding extremal problem in the multiplicative semigroup $M(S)$ of $S$-smooth numbers, i.e.,
\begin{equation*}
  \begin{aligned}
    \FS(X)\coloneq \max\bigl\{\,|A|: \  & A\subseteq M(S)\cap [1,X],                              \\
                                        & A\text{ contains no set }\{n,p_1n,\dots,p_r n\}\bigr\}.
  \end{aligned}
\end{equation*}
If $\ds_1<\ds_2<\cdots$ is the increasing sequence of $S$-smooth numbers, define
\begin{equation*}
  \begin{aligned}
    \fS(k)\coloneq \max\bigl\{\,|A|: \  & A\subseteq\{\ds_1,\dots,\ds_k\},                         \\
                                        & A\text{ contains no set }\{n,p_1 n,\dots,p_r n\}\bigr\}.
  \end{aligned}
\end{equation*}
In the case $k=0$, write $\fS(0)=0$ conventionally.

In exponent coordinates, the problem becomes one of avoiding unit corners in a weighted lattice simplex.

\smallskip

Our main result is an asymptotic formula for $\FS(X)$ and $\fS(k)$.

\begin{theorem}\label{thm:intro-smooth}
  Let $S=\{p_1,\dots,p_r\}$ be a finite set of distinct primes, $r\ge 2$. Then
  \[
    \FS(X)=\frac{r}{r+1}\PsiS(X)+O_S\bigl((\log X)^{r-1}\bigr)
  \]
  as $X\to\infty$. Moreover,
  \[
    \PsiS(X)=\frac{(\log X)^r}{r!\prod_{i=1}^r \log p_i}+O_S\bigl((\log X)^{r-1}\bigr),
  \]
  so that
  \[
    \FS(X)=\frac{r}{r+1}\cdot \frac{(\log X)^r}{r!\prod_{i=1}^r \log p_i}+O_S\bigl((\log X)^{r-1}\bigr).
  \]
  Equivalently,
  \[
    \fS(k)=\frac{r}{r+1}k+O_S\bigl(k^{(r-1)/r}\bigr).
  \]
\end{theorem}

The proof combines a modular coloring argument for the lower bound with an identity for corner-free sets in the weighted lattice simplex, which yields the upper bound.

\medskip

We also consider the connection with the problem on the full interval $[1,N]$, where
\[
  G_S(N)\coloneq \max\bigl\{|A|:A\subseteq [1,N],\ A\text{ contains no set }\{n,p_1n,\dots,p_rn\}\bigr\}.
\]
This leads to a representation of the associated density constant $\alpha_S$, and hence to explicit computable bounds.

Denote
\[
  P_S\coloneq \prod_{p\in S} p,\qquad \cS\coloneq \frac{\varphi(P_S)}{P_S}=\prod_{p\in S}\left(1-\frac1p\right).
\]

\begin{corollary}\label{cor:intro-full-density}
  Let $S=\{p_1,\dots,p_r\}$ be a finite set of distinct primes, $r\ge 2$. Then there exists a constant $\alphaS$ such that
  \[
    \GS(N)=\alphaS N+O_S\bigl((\log N)^r\bigr)
  \]
  as $N\to\infty$. In particular,
  \[
    \alphaS=\cS\sum_{k\ge 1}\frac{\Deltaf_S(k)}{\ds_k},\qquad \Deltaf_S(k)\coloneq \fS(k)-\fS(k-1)\in\{0,1\}.
  \]
\end{corollary}

The representation in \Cref{cor:intro-full-density} may also be deduced from \cite[Theorem 2]{GWS} by applying that result to the coefficient set $\{1,p_1,\dots,p_r\}$ and passing from hitting sets to configuration-free sets by complementation. For the convenience of the reader, we include a direct derivation in the present notation.

In turn, this representation yields nested computable bounds for $\alpha_S$.

\begin{corollary}\label{cor:intro-bounds}
  For $N\ge 1$, let
  \[
    L_{S,N}\coloneq \cS\sum_{k=1}^N \frac{\Deltaf_S(k)}{\ds_k},
    \qquad
    U_{S,N}\coloneq L_{S,N}+\cS\sum_{k>N}\frac1{\ds_k}.
  \]
  Then
  $
    L_{S,N}\le \alphaS\le U_{S,N},
  $
  and the intervals are nested, i.e.,
  \[
    L_{S,N}\le L_{S,N+1}\le \alphaS\le U_{S,N+1}\le U_{S,N}.
  \]
  Their width is exactly
  \[
    U_{S,N}-L_{S,N}=\cS\sum_{k>N}\frac1{\ds_k}.
  \]
\end{corollary}

In particular, the decimal expansion of $\alphaS$ can be computed to arbitrary precision.

\smallskip

Our next result shows that rational reciprocal sums over $S$-smooth denominators do not necessarily arise from eventually periodic binary sequences. In particular, eventual periodicity is not a necessary condition for rationality in this class of reciprocal sums.

\begin{proposition}\label{prop:intro-nonperiodic}
  Let $S$ be any finite set of primes with $|S|\ge 2$, and let
  $
    \ds_1<\ds_2<\cdots
  $
  be the increasing sequence of positive $S$-smooth numbers. Then there exists a binary sequence $(c_k)_{k\ge 1}$ such that
  \[
    \sum_{k=1}^\infty \frac{c_k}{\ds_k}\in \Q,
  \]
  but $(c_k)$ is not eventually periodic.
\end{proposition}

\smallskip

Finally, in the classical case $S=\{2,3\}$ we recover the Graham--Spencer--Witsenhausen constant
\[
  \alpha=\alpha_{\{2,3\}},
\]
obtain an explicit tail formula and prove two propositions for the structure of optimal sets.

\section{\Ss-smooth numbers and weighted simplices}

Fix a finite prime set
\[
  S=\{p_1,\dots,p_r\},\qquad \lambda_i\coloneq \log p_i\quad (1\le i\le r).
\]
For $u\ge 0$ define the \emph{weighted lattice simplex}
\[
  \SigmaS(u)\coloneq \left\{a=(a_1,\dots,a_r)\in \Zg^r: \sum_{i=1}^r a_i\lambda_i\le u\right\}.
\]
The map
$
  (a_1,\dots,a_r)\longmapsto p_1^{a_1}\cdots p_r^{a_r}
$
identifies $\SigmaS(u)$ with the set of $S$-smooth integers not exceeding $e^u$; hence
$
  \PsiS(e^u)=|\SigmaS(u)|.
$

\smallskip

Denote
\[
  \kS\coloneq \frac{1}{r!\prod_{i=1}^r \lambda_i}.
\]

\begin{lemma}\label{lem:lattice-count}
  As $u\to\infty$,
  \[
    |\SigmaS(u)|=\kS u^r+O_S(u^{r-1}).
  \]
  In particular,
  \[
    |\SigmaS(u)|\asymp_S u^r.
  \]
\end{lemma}

\begin{proof}
  We proceed by induction on $r$. If $r=1$, then
  \[
    \SigmaS(u)=\{a_1\in\Zg:a_1\lambda_1\le u\},
  \]
  hence
  \[
    |\SigmaS(u)|=\left\lfloor\frac{u}{\lambda_1}\right\rfloor+1=\frac{u}{\lambda_1}+O(1).
  \]
  Assume the statement holds for $r-1\ge 1$ primes. Write
  \[
    S^-\coloneq \{p_1,\dots,p_{r-1}\}.
  \]
  Partitioning $\SigmaS(u)$ by the value of $a_r$ gives
  \begin{equation}\label{eq:partition-simplex}
    |\SigmaS(u)|=\sum_{j=0}^{\lfloor u/\lambda_r\rfloor} |\Sigma_{S^-}(u-j\lambda_r)|.
  \end{equation}
  By the inductive hypothesis,
  \[
    |\Sigma_{S^-}(v)|=\frac{v^{r-1}}{(r-1)!\prod_{i=1}^{r-1}\lambda_i}+O_S(v^{r-2}).
  \]
  Substituting into \eqref{eq:partition-simplex} yields
  \[
    |\SigmaS(u)|
    =\frac{1}{(r-1)!\prod_{i=1}^{r-1}\lambda_i}\sum_{j=0}^{\lfloor u/\lambda_r\rfloor}(u-j\lambda_r)^{r-1}
    +O_S\!\left(\sum_{j=0}^{\lfloor u/\lambda_r\rfloor}(u-j\lambda_r)^{r-2}\right).
  \]
  The error sum is $O_S(u^{r-1})$. For the main sum, set $M\coloneq \lfloor u/\lambda_r\rfloor$. Then
  \[
    \sum_{j=0}^{M}(u-j\lambda_r)^{r-1}
    =\frac1{\lambda_r}\int_0^u v^{r-1}\,dv+O_S(u^{r-1})
    =\frac{u^r}{r\lambda_r}+O_S(u^{r-1}),
  \]
  since the sum is a Riemann sum for the monotone function $v^{r-1}$. Therefore
  \[
    |\SigmaS(u)|
    =\frac{u^r}{r!\prod_{i=1}^r \lambda_i}+O_S(u^{r-1}),
  \]
  as required.
\end{proof}

\begin{defn}
  For $a\in\Zg^r$, the \emph{unit corner at $a$} is the set
  \[
    \{a,a+e_1,\dots,a+e_r\},
  \]
  where $e_1,\dots,e_r$ are the standard basis vectors of $\R^r$.
\end{defn}

Under the map
\[
  (a_1,\dots,a_r)\longmapsto p_1^{a_1}\cdots p_r^{a_r},
\]
the forbidden configuration
\[
  \{n,p_1 n,\dots,p_r n\}
\]
corresponds to the unit corner
\[
  \{a,a+e_1,\dots,a+e_r\}.
\]
Thus the extremal problem for $\FS(X)$ is equivalent to the problem of choosing the largest subset of $\SigmaS(\log X)$ containing no unit corner.

\section{The extremal problem on \Ss-smooth numbers}

For $u\ge 0$, denote
\[
  g_S(u)\coloneq \max\bigl\{|A|:A\subseteq \SigmaS(u),\ A\text{ contains no unit corner}\bigr\}.
\]
Note that we have
$
  \FS(e^u)=g_S(u).
$

\subsection{Lower bound}

Define the coloring
\[
  \chi:\Zg^r\to \Z/(r+1)\Z,
  \qquad
  \chi(a_1,\dots,a_r)\equiv a_1+2a_2+\cdots+ra_r\pmod{r+1}.
\]

\begin{lemma}\label{lem:corner-colors}
  Every unit corner contains precisely one point of each color in $\Z/(r+1)\Z$.
\end{lemma}

\begin{proof}
  If $\chi(a)=c$, then
  \[
    \chi(a+e_i)\equiv c+i\pmod{r+1}\qquad (1\le i\le r).
  \]
  Therefore, the colors on the unit corner at $a$ are
  \[
    c,c+1,\dots,c+r\pmod{r+1},
  \]
  which are the residue classes modulo $r+1$.
\end{proof}

\begin{proposition}\label{prop:lower-bound}
  As $u\to\infty$,
  \[
    g_S(u)\ge \frac{r}{r+1}|\SigmaS(u)|+O_S(u^{r-1}).
  \]
\end{proposition}

\begin{proof}
  For each residue class $j\in \Z/(r+1)\Z$, denote
  \[
    C_j(u)\coloneq \{a\in\SigmaS(u):\chi(a)=j\}.
  \]
  By \Cref{lem:corner-colors}, removing any one color class destroys every unit corner. Hence
  \[
    g_S(u)\ge |\SigmaS(u)|-\min_j |C_j(u)|.
  \]
  Since the $r+1$ color classes partition $\SigmaS(u)$,
  \[
    \min_j |C_j(u)|\le \frac{1}{r+1}|\SigmaS(u)|,
  \]
  which already gives
  \[
    g_S(u)\ge \frac{r}{r+1}|\SigmaS(u)|.
  \]
  To obtain the sharper error term, partition $\SigmaS(u)$ into lines parallel to the $e_r$-axis. On each such line, the values of $\chi$ form an arithmetic progression modulo $r+1$, so the counts of the colors differ by at most $1$. The number of such lines is the number of lattice points in an $(r-1)$-dimensional weighted simplex, hence $O_S(u^{r-1})$ by \Cref{lem:lattice-count}. It follows that
  \[
    |C_j(u)|=\frac{1}{r+1}|\SigmaS(u)|+O_S(u^{r-1})
  \]
  for every $j$, and therefore
  \[
    g_S(u)\ge \frac{r}{r+1}|\SigmaS(u)|+O_S(u^{r-1}).\qedhere
  \]
\end{proof}

\subsection{Upper bound}

For a set $A\subseteq\SigmaS(u)$ and an index $1\le i\le r$, define the backward shift of $A$ in direction $e_i$ by
\[
  A^{(i)}\coloneq \{x\in \Zg^r:x+e_i\in A\}.
\]
Since $\SigmaS(u)$ is downward closed with respect to the coordinatewise order, we have $A^{(i)}\subseteq \SigmaS(u)$.

\begin{proposition}\label{prop:exact-shift}
  Let $A\subseteq \SigmaS(u)$ contain no unit corner. For an index $1\le i\le r$, set
  $
    Z_i(A)\coloneq |\{a\in A:a_i=0\}|.
  $
  For $x\in\SigmaS(u)$ define
  \[
    m_A(x)\coloneq \1_A(x)+\sum_{i=1}^r \1_{A^{(i)}}(x)
  \]
  and write
  \[
    D(A)\coloneq \sum_{x\in\SigmaS(u)} \bigl(r-m_A(x)\bigr).
  \]
  Then $D(A)\ge 0$ and
  \begin{equation}\label{eq:exact-shift}
    (r+1)|A|=r|\SigmaS(u)|+\sum_{i=1}^r Z_i(A)-D(A).
  \end{equation}
  Equivalently,
  \[
    |A|=\frac{r}{r+1}|\SigmaS(u)|+\frac{1}{r+1}\sum_{i=1}^r Z_i(A)-\frac{1}{r+1}D(A).
  \]
\end{proposition}

\begin{proof}
  Since $A$ contains no unit corner, the intersection
  \[
    A\cap A^{(1)}\cap\cdots\cap A^{(r)}
  \]
  is empty. Therefore $m_A(x)\le r$ for any $x\in\SigmaS(u)$, so $D(A)\ge 0$.

  Summing $m_A(x)$ over $\SigmaS(u)$ gives
  \[
    \sum_{x\in\SigmaS(u)} m_A(x)=|A|+\sum_{i=1}^r |A^{(i)}|.
  \]
  By the definition of $D(A)$,
  \[
    D(A)=r|\SigmaS(u)|-\left(|A|+\sum_{i=1}^r |A^{(i)}|\right).
  \]
  Furthermore, the map $x\mapsto x+e_i$ is a bijection from $A^{(i)}$ onto $\{a\in A:a_i\ge 1\}$; hence
  \[
    |A^{(i)}|=|A|-Z_i(A)\qquad (1\le i\le r).
  \]
  Substituting this into the previous identity yields
  \[
    D(A)=r|\SigmaS(u)|-(r+1)|A|+\sum_{i=1}^r Z_i(A),
  \]
  which is equivalent to \eqref{eq:exact-shift} and completes the proof.
\end{proof}

\begin{corollary}\label{cor:upper-bound}
  As $u\to\infty$,
  \[
    g_S(u)\le \frac{r}{r+1}|\SigmaS(u)|+O_S(u^{r-1}).
  \]
  Moreover,
  \[
    g_S(u)\le \frac{r}{r+1}|\SigmaS(u)|+\frac{1}{r+1}\sum_{i=1}^r F_i(u),
  \]
  where
  $
    F_i(u)\coloneq |\{a\in\SigmaS(u):a_i=0\}|.
  $
\end{corollary}

\begin{proof}
  Let $A\subseteq\SigmaS(u)$ contain no unit corner. By \Cref{prop:exact-shift},
  \[
    |A|\le \frac{r}{r+1}|\SigmaS(u)|+\frac{1}{r+1}\sum_{i=1}^r Z_i(A).
  \]
  Since $Z_i(A)\le F_i(u)$, it remains to estimate $F_i(u)$. Fixing $a_i=0$ identifies the corresponding face with an $(r-1)$-dimensional weighted simplex, so \Cref{lem:lattice-count} gives
  \[
    F_i(u)=O_S(u^{r-1}).
  \]
  Taking the maximum over $A$ proves the result.
\end{proof}

\smallskip

\begin{proof}[Proof of \Cref{thm:intro-smooth}]
  Combining \Cref{prop:lower-bound,cor:upper-bound} yields
  \[
    g_S(u)=\frac{r}{r+1}|\SigmaS(u)|+O_S(u^{r-1}).
  \]
  Since \Cref{lem:lattice-count} gives
  \[
    |\SigmaS(u)|=\frac{u^r}{r!\prod_{i=1}^r \log p_i}+O_S(u^{r-1})
  \]
  and we have $\FS(e^u)=g_S(u)$, we obtain
  \[
    \FS(e^u)=\frac{r}{r+1}\cdot\frac{u^r}{r!\prod_{i=1}^r \log p_i}+O_S(u^{r-1}).
  \]
  Fixing $u=\log X$ yields the first two asymptotic formulas.

  Now let $\ds_1<\ds_2<\cdots$ be the increasing sequence of positive $S$-smooth numbers, and let $\ds_k\le X<\ds_{k+1}$. Then
  \[
    \PsiS(X)=k,\qquad \FS(X)=\fS(k).
  \]
  Moreover,
  \[
    k=\PsiS(X)=\kS (\log X)^r+O_S\bigl((\log X)^{r-1}\bigr),
  \]
  so $\log X\asymp_S k^{1/r}$. Substituting into the $X$-asymptotic gives
  \[
    \fS(k)=\frac{r}{r+1}k+O_S\bigl((\log X)^{r-1}\bigr)=\frac{r}{r+1}k+O_S\bigl(k^{(r-1)/r}\bigr).\qedhere
  \]
\end{proof}

\section{The extremal problem on the full interval}
We now pass from the extremal problem inside $M(S)$ to the corresponding problem on $[1,N]$. The resulting density formula is consistent with the classical results for forbidden multiplicative configurations; we include a direct derivation in the present notation for completeness and because it leads naturally to effective tail bounds and computable intervals for $\alpha_S$.

For a positive integer $n$, we write
\[
  n=md,
  \qquad d\in M(S),
  \qquad (m,P_S)=1.
\]
This decomposition is unique.

\begin{proposition}\label{prop:component-decomposition}
  For every $N\ge 1$,
  \[
    \GS(N)=\sum_{\substack{m\le N\\ (m,P_S)=1}} \fS\bigl(\PsiS(N/m)\bigr).
  \]
\end{proposition}

\begin{proof}
  For each $m\le N$ with $(m,P_S)=1$, define
  $
    \mathcal C_m(N)\coloneq \{md:d\in M(S),\ md\le N\}.
  $
  The sets $\mathcal C_m(N)$ are pairwise disjoint and their union is $[1,N]$. Moreover, multiplying by a prime from $S$ does not change the $S$-free part $m$, so every configuration
  \[
    \{n,p_1n,\dots,p_r n\}
  \]
  is contained in a single component $\mathcal C_m(N)$. Thus, the extremal problem splits over the components as follows:
  \[
    \GS(N)=\sum_{\substack{m\le N\\ (m,P_S)=1}} \max\bigl\{|A|:A\subseteq \mathcal C_m(N),\ A\text{ contains no }\{n,p_1n,\dots,p_r n\}\bigr\}.
  \]
  $\mathcal C_m(N)$ is bijectively identified with the set of $S$-smooth numbers not exceeding $N/m$, and has size $\PsiS(N/m)$; hence, its contribution is $\fS(\PsiS(N/m))$. Summing finishes the proof.
\end{proof}

\begin{proof}[Derivation of \Cref{cor:intro-full-density}]
  By \Cref{prop:component-decomposition},
  \[
    \GS(N)=\sum_{\substack{m\le N\\ (m,P_S)=1}} \fS\bigl(\PsiS(N/m)\bigr).
  \]
  Since
  \[
    \fS(t)=\sum_{k=1}^t \Deltaf_S(k),
  \]
  we obtain
  \[
    \fS\bigl(\PsiS(N/m)\bigr)=\sum_{k\ge 1} \Deltaf_S(k)\,\1_{\{\ds_k\le N/m\}}.
  \]
  Therefore
  \[
    \GS(N)=\sum_{k\ge 1} \Deltaf_S(k)\,C_S\!\left(\frac{N}{\ds_k}\right),
  \]
  where
  $
    C_S(x)\coloneq |\{m\le x:(m,P_S)=1\}|.
  $
  By inclusion--exclusion,
  \[
    C_S(x)=\cS x+O_S(1).
  \]
  Substituting this gives
  \[
    \GS(N)=\cS N\sum_{\ds_k\le N}\frac{\Deltaf_S(k)}{\ds_k}+O_S\bigl(\PsiS(N)\bigr).
  \]
  Since
  \[
    \sum_{k\ge 1}\frac1{\ds_k}=\sum_{d\in M(S)} \frac1d=\prod_{p\in S}\left(1-\frac1p\right)^{-1}<\infty,
  \]
  the series
  \[
    \sum_{k\ge 1}\frac{\Deltaf_S(k)}{\ds_k}
  \]
  converges absolutely. Define
  \[
    \alphaS\coloneq \cS\sum_{k\ge 1}\frac{\Deltaf_S(k)}{\ds_k}.
  \]
  Then
  \[
    \GS(N)-\alphaS N
    =O_S\bigl(\PsiS(N)\bigr)-\cS N\sum_{\ds_k>N}\frac{\Deltaf_S(k)}{\ds_k}.
  \]
  By \Cref{lem:lattice-count},
  \[
    \PsiS(N)=O_S\bigl((\log N)^r\bigr).
  \]
  Additionally,
  \[
    0\le N\sum_{\ds_k>N}\frac{\Deltaf_S(k)}{\ds_k}
    \le N\sum_{\ds_k>N}\frac1{\ds_k}=O_S\bigl((\log N)^{r-1}\bigr)
  \]
  by \Cref{prop:tail-asymptotic} below. Hence
  \[
    \GS(N)=\alphaS N+O_S\bigl((\log N)^r\bigr).\qedhere
  \]
\end{proof}

\smallskip

\begin{proof}[Proof of \Cref{cor:intro-bounds}]
  Since $0\le \Deltaf_S(k)\le 1$ for every $k$,
  \[
    0\le \alphaS-L_{S,N}=\cS\sum_{k>N}\frac{\Deltaf_S(k)}{\ds_k}
    \le \cS\sum_{k>N}\frac1{\ds_k}=U_{S,N}-L_{S,N}.
  \]
  This proves
  \[
    L_{S,N}\le \alphaS\le U_{S,N}.
  \]
  Furthermore,
  \[
    L_{S,N+1}-L_{S,N}=\frac{\cS\,\Deltaf_S(N+1)}{\ds_{N+1}}\ge 0,
  \]
  and
  \[
    U_{S,N}-U_{S,N+1}=\frac{\cS\,(1-\Deltaf_S(N+1))}{\ds_{N+1}}\ge 0.
  \]
  Thus, the intervals are nested. The formula for the width is immediate from the definition.
\end{proof}

For $X\ge 1$, define the reciprocal tail
\[
  T_S(X)\coloneq \sum_{\substack{d\in M(S)\\ d>X}} \frac1d.
\]
Then
\[
  U_{S,N}-L_{S,N}=\cS T_S(\ds_N).
\]

\smallskip

The following proposition gives a method of computation of the reciprocal tail. As the quantities $L_{S,N}$ and $U_{S,N}$ are computable for each fixed $N$, choosing $N$ so that $U_{S,N}-L_{S,N}<10^{-M}$ determines the first $M$ decimal digits of $\alphaS$.

\begin{proposition}\label{prop:tail-recursion}
  Let $S=\{p_1,\dots,p_r\}$ be a finite set of distinct primes, and denote $S'\coloneq S\setminus\{p_1\}$. If
  $
    A\coloneq \lfloor \log_{p_1} X\rfloor,
  $
  then
  \[
    T_S(X)=\sum_{a=0}^{A}\frac1{p_1^a}T_{S'}\!\left(\frac{X}{p_1^a}\right)
    +\frac{p_1^{-A}}{p_1-1}\prod_{q\in S'}\left(1-\frac1q\right)^{-1}.
  \]
  In particular, for $S=\{p\}$,
  \[
    T_{\{p\}}(X)=\frac{p^{-\lfloor \log_p X\rfloor}}{p-1}.
  \]
\end{proposition}

\begin{proof}
  Every element of $M(S)$ can be written uniquely in the form $p_1^a e$, where $a\ge 0$ and $e\in M(S')$. Hence
  \[
    T_S(X)=\sum_{\substack{a\ge 0,\ e\in M(S')\\ p_1^a e>X}} \frac{1}{p_1^a e}.
  \]
  If $0\le a\le A$, then $p_1^a\le X$, so the condition $p_1^a e>X$ is equivalent to
  \[
    e>\frac{X}{p_1^a},
  \]
  and the contribution of these terms is
  \[
    \sum_{a=0}^{A} \frac1{p_1^a}T_{S'}\!\left(\frac{X}{p_1^a}\right).
  \]
  If $a\ge A+1$, then $p_1^a>X$, so every $e\in M(S')$ is admissible. The remaining contribution is
  \[
    \begin{aligned}
      \sum_{a=A+1}^{\infty} \frac1{p_1^a}\sum_{e\in M(S')}\frac1e
       & =\left(\sum_{a=A+1}^{\infty} \frac1{p_1^a}\right)
      \prod_{q\in S'}\left(1-\frac1q\right)^{-1}           \\
       & =\frac{p_1^{-A}}{p_1-1}
      \prod_{q\in S'}\left(1-\frac1q\right)^{-1}.
    \end{aligned}
  \]
  This proves the recursion. The one-prime formula is the special case
  \[
    T_{\{p\}}(X)=\sum_{a>\lfloor \log_p X\rfloor} \frac1{p^a}.\qedhere
  \]
\end{proof}

\begin{proposition}\label{prop:tail-asymptotic}
  As $X\to\infty$,
  \[
    T_S(X)=\frac{(\log X)^{r-1}}{(r-1)!\,X\prod_{p\in S} \log p}+O_S\!\left(\frac{(\log X)^{r-2}}{X}\right).
  \]
\end{proposition}

\begin{proof}
  By partial summation,
  \[
    T_S(X)=-\frac{\PsiS(X)}{X}+\int_X^{\infty} \frac{\PsiS(t)}{t^2}\,dt.
  \]
  By \Cref{lem:lattice-count},
  \[
    \PsiS(t)=\kS (\log t)^r+O_S\bigl((\log t)^{r-1}\bigr),
    \qquad
    \kS=\frac{1}{r!\prod_{p\in S}\log p}.
  \]
  Therefore,
  \[
    T_S(X)
    =-\kS\frac{(\log X)^r}{X}
    +\kS\int_X^{\infty}\frac{(\log t)^r}{t^2}\,dt
    +O_S\!\left(\frac{(\log X)^{r-1}}{X}\right).
  \]
  With the substitution $t=e^u$, we obtain
  \[
    \begin{aligned}
      \int_X^{\infty} \frac{(\log t)^r}{t^2}\,dt
       & =\int_{\log X}^{\infty} u^r e^{-u}\,du                             \\
       & =\frac{(\log X)^r+r(\log X)^{r-1}+O\bigl((\log X)^{r-2}\bigr)}{X}.
    \end{aligned}
  \]
  The leading terms cancel, and since $r\kS=((r-1)!\prod_{p\in S}\log p)^{-1}$, the stated asymptotic follows.
\end{proof}

\section{Non-periodic rational reciprocal sums}

\begin{proof}[Proof of \Cref{prop:intro-nonperiodic}]
  Choose distinct primes $p,q\in S$. Define
  \[
    c_k\coloneq
    \begin{cases}
      1, & \text{if }\ds_k=p^m\text{ for some }m\ge 1, \\
      0, & \text{otherwise.}
    \end{cases}
  \]
  Then
  \[
    \sum_{k=1}^\infty \frac{c_k}{\ds_k}=\sum_{m=1}^\infty \frac1{p^m}=\frac{1}{p-1}\in\Q.
  \]
  It remains to show that $(c_k)$ is not eventually periodic.

  Let $k_m$ denote the index such that $\ds_{k_m}=p^m$. For each integer
  \[
    1\le b\le \left\lfloor \frac{m}{\log_p q}\right\rfloor,
  \]
  set
  \[
    a_b\coloneq \lceil m-b\log_p q\rceil.
  \]
  Because $p$ and $q$ are distinct primes, the ratio $\log_p q=\log q/\log p$ is irrational; otherwise $p^u=q^v$ for some positive integers $u,v$, which is impossible. Therefore, the number $m-b\log_p q$ is not an integer, and
  \[
    m-b\log_p q<a_b<m-b\log_p q+1.
  \]
  Adding $b\log_p q$ gives
  \[
    m<a_b+b\log_p q<m+1.
  \]
  Exponentiating with base $p$ yields
  \[
    p^m<p^{a_b}q^b<p^{m+1}.
  \]
  These are distinct $S$-smooth integers lying strictly between $p^m$ and $p^{m+1}$; hence, the interval $(p^m,p^{m+1})$ contains at least $\lfloor m/\log_p q\rfloor$ distinct $S$-smooth numbers, so
  \[
    k_{m+1}-k_m\to\infty.
  \]
  In an eventually periodic binary sequence with infinitely many $1$'s, the gaps between successive $1$'s are eventually bounded by the period. Since the gaps here are unbounded, $(c_k)$ is not eventually periodic.
\end{proof}

\section{The classical case}

Throughout this section we specialize to $S=\{2,3\}$. Denote
\[
  \alpha\coloneq \alpha_{\{2,3\}},
\]
let
$
  d_1<d_2<\cdots
$
be the increasing sequence of positive $3$-smooth numbers, and write $f(k)\coloneq f_{\{2,3\}}(k)$.

Since
\[
  \cS=\left(1-\frac12\right)\left(1-\frac13\right)=\frac13,
\]
\Cref{cor:intro-full-density} gives
\[
  \alpha=\frac13\sum_{k\ge 1} \frac{\Deltaf(k)}{d_k},
  \qquad \Deltaf(k)\coloneq f(k)-f(k-1)\in\{0,1\}.
\]
The nested bounds in \Cref{cor:intro-bounds} become
\[
  L_N\coloneq \frac13\sum_{k=1}^N \frac{\Deltaf(k)}{d_k},
  \qquad
  U_N\coloneq L_N+\frac13\sum_{k>N}\frac1{d_k},
\]
with
\[
  L_N\le L_{N+1}\le \alpha\le U_{N+1}\le U_N.
\]

The general tail recursion becomes explicit.

\begin{corollary}\label{cor:explicit-tail-23}
  Let $X\ge 1$ and denote
  \[
    A\coloneq \lfloor \log_2 X\rfloor,
    \qquad
    b_a\coloneq 1+\left\lfloor \log_3\left(\frac{X}{2^a}\right)\right\rfloor
    \qquad (0\le a\le A).
  \]
  Then
  \[
    \sum_{\substack{u,v\ge 0\\ 2^u 3^v>X}} \frac{1}{2^u3^v}
    =\frac32\sum_{a=0}^{A}\frac{1}{2^a 3^{b_a}}+\frac{3}{2^{A+1}}.
  \]
  Consequently,
  \[
    U_N-L_N=\frac13\left(\frac32\sum_{a=0}^{A}\frac{1}{2^a 3^{b_a}}+\frac{3}{2^{A+1}}\right)
  \]
  when $X=d_N$.
\end{corollary}

\begin{proof}
  Apply \Cref{prop:tail-recursion} with $p_1=2$ and $S'=\{3\}$. Since
  \[
    T_{\{3\}}(Y)=\frac{3^{-\lfloor \log_3 Y\rfloor}}{2}=\frac{3}{2\cdot 3^{b_a}}
    \qquad\text{for }Y=\frac{X}{2^a},
  \]
  the stated identity follows.
\end{proof}

\begin{proposition}\label{prop:no-nested}
  There is no nested optimal family
  \[
    A_1\subseteq A_2\subseteq A_3\subseteq \cdots,
    \qquad
    A_k\subseteq \{d_1,\dots,d_k\},
    \qquad
    |A_k|=f(k),
  \]
  with every $A_k$ free of sets of the form $\{x,2x,3x\}$.
\end{proposition}

\begin{proof}
  The initial values are
  \[
    f(1)=1,\qquad f(2)=2,\qquad f(3)=2,\qquad f(4)=3,\qquad f(5)=4,
  \]
  and
  \[
    \{d_1,d_2,d_3,d_4,d_5\}=\{1,2,3,4,6\}.
  \]
  The unique optimal subset of size $4$ is
  \[
    \{1,3,4,6\}.
  \]
  Indeed, every other $4$-subset contains either $\{1,2,3\}$ or $\{2,4,6\}$.
  Hence a nested optimal family would have to satisfy
  \[
    A_5=\{1,3,4,6\}.
  \]
  Since $A_4\subseteq \{1,2,3,4\}$ and $|A_4|=3$, this forces
  \[
    A_4=\{1,3,4\}.
  \]
  Likewise $A_3\subseteq \{1,2,3\}$ and $|A_3|=2$, so $A_3=\{1,3\}$. But $A_2\subseteq \{1,2\}$ and $|A_2|=2$ implies
  \[
    A_2=\{1,2\},
  \]
  which is not contained in $A_3$ -- a contradiction.
\end{proof}

\begin{proposition}\label{prop:local-update}
  Let $d_k=2^a3^b$.
  \begin{enumerate}
    \item If $b=0$, then adding $d_k$ to the previous prefix creates no new forbidden triple, and therefore
          \[
            \Deltaf(k)=1.
          \]
    \item If $b>0$, then the only new forbidden triple created at stage $k$ is
          \[
            \{2^a3^{b-1},2^{a+1}3^{b-1},2^a3^b\}.
          \]
          Consequently,
          \[
            \Deltaf(k)=1
          \]
          if and only if there exists an optimal subset of $\{d_1,\dots,d_{k-1}\}$ omitting at least one of the two lower points
          \[
            2^a3^{b-1},\qquad 2^{a+1}3^{b-1}.
          \]
  \end{enumerate}
\end{proposition}

\begin{proof}
  Write the new point in exponent coordinates as $(a,b)$. A forbidden triple corresponds to a corner
  \[
    \{(x,y),(x+1,y),(x,y+1)\}.
  \]
  A corner containing $(a,b)$ can occur in at most three ways.

  If $(a,b)$ is the anchor, then the corner is
  \[
    \{(a,b),(a+1,b),(a,b+1)\},
  \]
  and both other points correspond to larger $3$-smooth numbers, so neither belongs to the previous prefix.

  If $(a,b)$ is the east point, then the corner is
  \[
    \{(a-1,b),(a,b),(a-1,b+1)\}.
  \]
  The point $(a-1,b+1)$ has larger weight
  \[
    (a-1)\log 2+(b+1)\log 3>a\log 2+b\log 3,
  \]
  so it is not yet present.

  If $(a,b)$ is the north point, then the corner is
  \[
    \{(a,b-1),(a+1,b-1),(a,b)\}.
  \]
  When $b>0$, both lower points have smaller weight and therefore already belong to the previous prefix. This is the unique new forbidden triple.

  If $b=0$, the third case does not occur, so no new forbidden triple is created and $\Deltaf(k)=1$.

  Assume now that $b>0$. If there is an optimal subset of the previous prefix omitting one of the two lower points, then adjoining $d_k$ preserves the forbidden-triple-free property, so the optimum increases by $1$. Conversely, if $\Deltaf(k)=1$, then some optimal set at stage $k$ contains the new point $d_k$; removing it leaves an optimal set at stage $k-1$ that omits at least one of the two lower points, otherwise the unique new triple would already be present.
\end{proof}

\section*{Supplementary material}
The code and data used for the numerical calculations in this paper are available at
\url{https://github.com/nikolaveselinov/s-smooth-forbidden-configurations}.

\nocite{*}
\bibliographystyle{plain}
\bibliography{refs}

\end{document}